\documentclass[11pt]{article}
\usepackage[margin=1in]{geometry} 
\usepackage{amssymb,amsfonts,amsmath,bbm,mathrsfs,stmaryrd}
\usepackage{xcolor}
\usepackage{url}

\usepackage[colorlinks,
             linkcolor=black!75!red,
             citecolor=blue,
             pdftitle={On quantum symmetries of compact metric spaces},
             pdfauthor={Alexandru Chirvasitu},
             pdfproducer={pdfLaTeX},
             pdfpagemode=None,
             bookmarksopen=true
             bookmarksnumbered=true]{hyperref}

\usepackage{tikz}
\usetikzlibrary{arrows,decorations.pathreplacing,decorations.markings,shapes.geometric,through,fit,shapes.symbols,positioning}

\usepackage[amsmath,thmmarks,hyperref]{ntheorem}
\usepackage{cleveref}

\crefname{section}{Section}{Sections}
\crefformat{section}{#2Section~#1#3} 
\Crefformat{section}{#2Section~#1#3} 

\crefname{subsection}{\S}{\S\S}
\crefformat{subsection}{#2\S#1#3} 
\Crefformat{subsection}{#2\S#1#3} 

\theoremstyle{plain}

\newtheorem{lemma}{Lemma}[section]
\newtheorem{proposition}[lemma]{Proposition}
\newtheorem{corollary}[lemma]{Corollary}
\newtheorem{theorem}[lemma]{Theorem}
\newtheorem{conjecture}[lemma]{Conjecture}

\theoremstyle{nonumberplain}
\newtheorem{theoremN}{Theorem}
\newtheorem{propositionN}{Proposition}

\theoremstyle{plain}
\theorembodyfont{\upshape}
\theoremsymbol{\ensuremath{\blacklozenge}}

\newtheorem{definition}[lemma]{Definition}

\newtheorem{remark}[lemma]{Remark}
\newtheorem{convention}[lemma]{Convention}

\crefname{definition}{definition}{definitions}
\crefformat{definition}{#2definition~#1#3} 
\Crefformat{definition}{#2Definition~#1#3} 

\crefname{ex}{example}{examples}
\crefformat{example}{#2example~#1#3} 
\Crefformat{example}{#2Example~#1#3} 

\crefname{remark}{remark}{remarks}
\crefformat{remark}{#2remark~#1#3} 
\Crefformat{remark}{#2Remark~#1#3} 

\crefname{convention}{convention}{conventions}
\crefformat{convention}{#2convention~#1#3} 
\Crefformat{convention}{#2Convention~#1#3}

\crefname{lemma}{lemma}{lemmas}
\crefformat{lemma}{#2lemma~#1#3} 
\Crefformat{lemma}{#2Lemma~#1#3} 

\crefname{proposition}{proposition}{propositions}
\crefformat{proposition}{#2proposition~#1#3} 
\Crefformat{proposition}{#2Proposition~#1#3} 

\crefname{corollary}{corollary}{corollaries}
\crefformat{corollary}{#2corollary~#1#3} 
\Crefformat{corollary}{#2Corollary~#1#3} 

\crefname{theorem}{theorem}{theorems}
\crefformat{theorem}{#2theorem~#1#3} 
\Crefformat{theorem}{#2Theorem~#1#3} 

\crefname{assumption}{assumption}{Assumptions}
\crefformat{assumption}{#2assumption~#1#3} 
\Crefformat{assumption}{#2Assumption~#1#3} 

\crefname{equation}{}{}
\crefformat{equation}{(#2#1#3)} 
\Crefformat{equation}{(#2#1#3)}

\theoremstyle{nonumberplain}
\theoremsymbol{\ensuremath{\blacksquare}}

\newtheorem{proof}{Proof}
\newtheorem{proof of main}{Proof of \Cref{th.main}}
\newtheorem{proof of Hall}{Proof of \Cref{th.Hall}}
\newtheorem{proof of sbgp}{Proof of \Cref{th.sbgp}}

\newtheorem{proof of res1}{Proof of \Cref{th.res1}}
\newtheorem{proof of res2}{Proof of \Cref{th.res2}}
\newtheorem{proof of key}{Proof of \Cref{pr.key}}

\newcommand\bC{{\mathbb C}}

\newcommand\bR{{\mathbb R}}

\newcommand\bZ{{\mathbb Z}}

\newcommand\cA{{\mathcal A}}
\newcommand\cC{{\mathcal C}}
\newcommand\cH{{\mathcal H}}
\newcommand\cM{{\mathcal M}}
\newcommand\cW{{\mathcal W}}

\DeclareMathOperator{\id}{id}

\DeclareMathOperator{\pr}{\mathrm{Prob}}

\DeclareMathOperator{\cact}{\cat{CAct}}

\newcommand{\define}[1]{{\em #1}}

\newcommand{\cat}[1]{\textsc{#1}}

\title{On quantum symmetries of compact metric spaces}
\author{Alexandru Chirvasitu\footnote{UC Berkeley, \url{chirvasitua@math.berkeley.edu}}}
\begin{document}

\date{}

\maketitle

\begin{abstract}
An action of a compact quantum group on a compact metric space $(X,d)$ is (D)-isometric if the distance function is preserved by a diagonal action on $X\times X$. We show that an isometric action in this sense has the following additional property: The corresponding action on the algebra of continuous functions on $X$ by the convolution semigroup of probability measures on the quantum group contracts Lipschitz constants. It is, in other words, isometric in another sense due to H. Li, J. Quaegebeur and M. Sabbe; this partially answers a question of D. Goswami.

We also introduce other possible notions of isometric quantum action, in terms of the Wasserstein $p$-distances between probability measures on $X$ for $p\ge 1$, used extensively in optimal transportation. It turns out all of these definitions of quantum isometry fit in a tower of implications, with the two above at the extreme ends of the tower. We conjecture that they are all equivalent.  
\end{abstract}


\tableofcontents

%%%%%%%%%%%%%%%%%%%%%%%%%%%%%%%%%%%%%%%%%%%%%%%%%%%%%%%%%%%%%%%%%%%%%%%%%%%%%%%%%%%%%%%%%%%%%%%%%%%%%%%%%%%%%%%%%%
%%%%%%%%%%%%%%%%%%%%%%%%%%%%%%%%%%%%%%%%%%%%%%%%%%%%%%%%%%%%%%%%%%%%%%%%%%%%%%%%%%%%%%%%%%%%%%%%%%%%%%%%%%%%%%%%%%
\section*{Introduction}

Operator-algebraic compact quantum groups were introduced in \cite{Wor87} as non-commutative analogues of the $C^*$-algebra of continuous functions on a compact group. The field has developed explosively since then, offering many interesting directions to pursue.  

One popular way of constructing examples has been to start out  with some algebraic or geometric structure, and consider the ``largest'' compact quantum group that acts so as to preserve it. Examples of such structures include operator algebras \cite{Boc95,Wan98,Wan99}, finite-dimensional Hilbert spaces (perhaps equipped with a bilinear form) \cite{DaeWan96}, finite graphs \cite{Bic03}, Riemannian manifolds \cite{BhoGos09}, etc. 

Surprisingly at the time, S. Wang discovered in \cite{Wan98} that even the ``half classical'' situation when a quantum group acts on an ordinary space (in that case a finite one) can lead to interesting and non-trivial objects: The quantum automorphism group of a finite set, introduced in that paper, is truly quantum (i.e. not the function algebra on a finite group) provided the set acted upon is not too small. 

In the same spirit, quantum automorphism groups of finite metric spaces were constructed in \cite{Ban05}. Banica's notion of what it means for a quantum group action to be isometric can be adapted to possibly infinite compact metric spaces. This is done in \cite{Gos12}, where D. Goswami studies the problem of whether or not a universal quantum isometric action exists in this infinite setting.   

This brings us to the topic of the paper: actions of compact quantum groups on (classical) compact metric spaces. 

Call a quantum group action on a compact metric space $(X,d)$ (D)-isometric if it satisfies the conditions from \cite{Gos12}. Roughly speaking, this entails that the distance function $d:X\times X\to\bR_{\ge 0}$ be invariant under a certain ``diagonal'' action of the quantum group on $X\times X$ (see \Cref{subse.prel_metric} for precise definitions).

One constant theme in non-commutative geometry is that the classical concepts to be generalized often do not have one, canonical, ``best'' formulation in the quantum or non-commutative setting. As a consequence, classical notions sometimes bi(or tri or multi)furcate into competing non-commutative analogues that are not always obviously equivalent. Quantum symmetries of compact metric spaces are a case in point; there are alternative attempts in the literature to define what it means for a quantum action on a compact metric space $(X,d)$ to be isometric. Another such proposal (apart from the one alluded to in the previous paragraph) appears in \cite{QuaSab12}. Again, see, \Cref{subse.prel_metric} for a more detailed discussion. 

Since a compact quantum group is by definition a $C^*$-algebra $A$ (see \Cref{def.cqg} below) with some additional structure, we can think of the states of $A$ as probability measures on the underlying non-commutative space of the quantum group. Moreover, the comultiplication $A\to A\otimes A$ induces an associative binary operation on the set of states $S(A)$, making it into a semigroup. 

An action of the compact quantum group $A$ on a compact space $X$ induces a left action of $S(A)$ on the set $\cC(X)$ of continuous functions on $X$, as well as a right action of $S(A)$ on the set $\pr(X)$ of probability measures on $X$. 

If $X$ is equipped with a metric $d$, write $L(f)$ for the value of the Lipschitz seminorm of $f\in\cC(X)$:
\[
 L(f) = \sup_{x\ne y}\frac{|f(x)-f(y)|}{d(x,y)}.  
\]
We call $L(f)$ the Lipschitz constant of $f$. It can be infinity, but the subspace of Lipschitz functions (those with finite Lipschitz constant) is dense in $\cC(X)$. We say that the action of the quantum group on $X$ is (Lip)-isometric if for any  $\varphi\in S(A)$ and any $f\in\cC(X)$ the Lipschitz constant of $\varphi\triangleright f$ is not larger than that of $f$. 

It is shown in \cite[Theorem 3.5]{QuaSab12} that quantum group actions on finite metric spaces are (Lip)-isometric if and only if they are (D)-isometric. The question of whether or not this is true  for general compact metric spaces is posed in \cite{Gos12}. We are not able to answer it fully, but one of the implications is part of \Cref{th.main}:

\begin{propositionN}
A (D)-isometric quantum action on a compact metric space is (Lip)-isometric. 
\end{propositionN}

We can say more though.

Embed $X\subset\pr(X)$ in the obvious way, by regarding points as Dirac delta measures. The distance function $d$ on $X$ can be extended to a metric on $\pr(X)$ in many ways so as to induce the weak$^*$ topology on probability measures. Some popular extensions of $d$ with this property are the so-called Wasserstein metrics on $\pr(X)$ (see \Cref{subse.prel_Was}): 

For $\mu,\nu\in\pr(X)$ and $p\ge 1$ define
\[
 W_p(\mu,\nu) = \inf_{\pi\in\Pi(\mu,\nu)}\left(\int_{X\times X} d(x,y)^p\ \text{d}\pi\right)^{\frac 1p},
\] 
where $\Pi(\mu,\nu)$ is the space of probability measures on $X\times X$ with marginals $\mu$ and $\nu$ on the two Cartesian factors $X$. It turns out that being (Lip)-isometric can be rephrased as follows: The action of any $\varphi\in S(A)$ on $\pr(X)$ contracts the $W_1$ metric. For this reason, we also say `(Lip$_1$)-isometric' instead. 

It is now compelling to define condition (Lip$_p$) for a quantum action on $(X,d)$ by analogy with (Lip$_1$): If $A$ is the $C^*$-algebra of the compact quantum group, simply demand that any $\varphi\in S(A)$ contract the metric $W_p$ on $\pr(X)$. 

The content of \Cref{cor.main} can now be phrased as follows:

\begin{theoremN}
For a quantum action on a compact metric space condition (Lip$_p$) is stronger with larger $p\ge 1$, and condition (D) is stronger than all (Lip$_p$). 
\end{theoremN}

As indicated above, universal objects are important sources of examples of compact quantum groups. One often looks for the largest quantum group acting faithfully (in some appropriate sense) on a certain object. It is often not clear when such quantum groups exist. In \cite{Wan98}, for instance, it transpired that a finite-dimensional $C^*$-algebra has a quantum automorphism group if it is commutative, but not otherwise. A finite-dimensional $C^*$-algebra equipped with a trace, on the other hand, always has a quantum automorphism group. 

As far as metric spaces are concerned, finite ones always have quantum automorphism groups \cite{Ban05}, as do those embeddable in some Euclidean space $\bR^n$ \cite{Gos12}. It is unclear whether all compact metric spaces have one though. 

Here we work with a somewhat weaker notion of universality, following \cite{QuaSab12}. One of the problems posed in that paper can be paraphrased as follows: Given a faithful action of a quantum group on a compact metric space, is there a largest quantum subgroup that acts isometrically? 

The question in \cite{QuaSab12} applies to (Lip)-isometric actions, and is not answered there fully. We show in \Cref{th.sbgp} below that in the (D)-isometric case the answer is always affirmative.  

We make a few conjectures related to all of this in \Cref{se.main,se.sbgp}.

The paper is organized as follows: 

\Cref{se.prel} contains a recollection of the main ingredients we need, as well as some auxiliary results. 

In \Cref{se.Hall} we prove a measure-theoretic version of Hall's marriage theorem on the existence of matchings in bipartite graphs. It generalizes Hall's and other combinatorial results of the same nature, and will be used later in the proof of \Cref{th.main}.  

\Cref{se.main} contains the proof of the main result, \Cref{th.main}. In addition, as a consequence of the discussion in that section, we show in \Cref{pr.injectivity} that the underlying map $\rho:\cC(X)\to\cC(X)\otimes A$ of a (D)-isometric coaction is one-to-one.   

Finally, in \Cref{se.sbgp} we prove that every quantum group acting faithfully on a compact metric space has a largest quantum subgroup acting (D)-isometrically (\Cref{th.sbgp}).

%%%%%%%%%%%%%%%%%%%%%%%%%%%%%%%%%%%%%%%%%%%%%%%%%%%%%%%%%%%%%%%%%%%%%%%%%%%%%%%%%%%%%%%%%%%%%%%%%%%%%%%%%%%%%%%%%%
\subsection*{Acknowledgements}

I am indebted to Debashish Goswami for valuable discussions on the contents of \cite{Gos12}.

%%%%%%%%%%%%%%%%%%%%%%%%%%%%%%%%%%%%%%%%%%%%%%%%%%%%%%%%%%%%%%%%%%%%%%%%%%%%%%%%%%%%%%%%%%%%%%%%%%%%%%%%%%%%%%%%%%
%%%%%%%%%%%%%%%%%%%%%%%%%%%%%%%%%%%%%%%%%%%%%%%%%%%%%%%%%%%%%%%%%%%%%%%%%%%%%%%%%%%%%%%%%%%%%%%%%%%%%%%%%%%%%%%%%%
\section{Preliminaries}\label{se.prel}

We deal below both with $C^*$ and von Neumann (or $W^*$) quantum groups, so we are assuming some basic background on operator algebras, as covered, say, in \cite{Tak02}. 

For a Hilbert space $\cH$ let $B(\cH)$ be the algebra of all bounded operators on $\cH$, equipped with the usual norm and $*$ operation. For our purposes it will be sufficient to recall that a $C^*$-algebra is a norm-closed $*$-subalgebra of $B(\cH)$ for some Hilbert space $\cH$. 

The Banach space $B(\cH)$ turns out to be the dual of a (unique) Banach space, and can hence be equipped with a weak$^*$ topology. A von Neumann or $W^*$-algebra is a weak$^*$-closed $*$-subalgebra of some $B(\cH)$ (in particular, it is also a $C^*$-algebra). 

We will not make much use of the usual abstract characterizations of $C^*$ and $W^*$-algebras (e.g. \cite[I.1.2,III.3.1]{Tak02}). Suffice it to say that a von Neumann algebra is a $C^*$-algebra which admits a predual (i.e. it is the continuous dual of a Banach space). Just as for $B(\cH)$, the predual is unique, so that this is a property rather than an additional piece of structure (\cite[III.3.9]{Tak02}). We use `von Neumann algebra' and `$W^*$-algebra' interchangeably, breaking with the more nuanced use in \cite{Tak02}.  

Throughout, a $C^*$ morphism between $C^*$-algebras is a norm continuous $*$-preserving algebra homomorphism, while a $W^*$ morphism between $W^*$-algebras is a norm continuous, $*$-preserving algebra homomorphism that is also \define{normal}, i.e. continuous with respect to the weak$^*$ topologies induced by the preduals.  

All algebras and homomorphisms in this paper are unital. 

The symbol `$\otimes$' changes meaning depending on what types of objects it is placed between. If $A$ and $B$ are $C^*$-algebras, then $A\otimes B$ denotes the minimal (or injective) tensor product (\cite[IV.4.8]{Tak02}). It is again a $C^*$-algebra, in which the ordinary, algebraic tensor product of $A$ and $B$ is norm-dense. 

If $A$ and $B$ are von Neumann algebras, then $\otimes$ is the spatial tensor product, as in \cite[IV.5.1]{Tak02}. It is a again a von Neumann algebra.   

On those rare occasions when $A$ and $B$ are just plain algebras and no topology is involved, $\otimes$ is the algebraic tensor product. 
 
We denote by $A^*$ the appropriate set of norm continuous functionals on $A$: All of them if $A$ is a $C^*$-algebra, and those that are additionally normal if $A$ is von Neumann. A \define{state} on a $C^*$ or $W^*$-algebra $A$ is a unital functional $\varphi\in A^*$, positive in the sense that $\varphi(a^*a)$ is non-negative for every $a\in A$, and normal (in the same sense as before) in the $W^*$ case. We write $S(A)$ for the set of states.  

Finally, `measure' always means finite Borel measure on a metrizable compact Hausdorff space. In particular, all measures are regular.

%%%%%%%%%%%%%%%%%%%%%%%%%%%%%%%%%%%%%%%%%%%%%%%%%%%%%%%%%%%%%%%%%%%%%%%%%%%%%%%%%%%%%%%%%%%%%%%%%%%%%%%%%%%%%%%%%%
\subsection{Quantum groups and actions}\label{subse.prel_cqg}

We use the same notion of $C^*$-algebraic compact quantum group as many of the authors we have cited, e.g. \cite{Wan98,QuaSab12,Gos12}. The survey \cite{KusTus99}, for instance, provides sufficient background. We will generally drop the word `compact' from `compact quantum group'. 

In the von Neumann algebra setting we simply replicate the definition with only obvious modifications.

\begin{definition}\label{def.cqg}
A $C^*$ (respectively $W^*$ or von Neumann) \define{quantum group} $(A,\Delta)$ is a unital $C^*$(resp. $W^*$)-algebra equipped with a $C^*$(resp. $W^*$)-algebra homomorphism $\Delta:A\to A\otimes A$ satisfying the following properties:
\begin{enumerate}
\renewcommand{\labelenumi}{(\arabic{enumi})}
 \item $\Delta$ is coassociative, in the sense that the diagram
   \[
     \tikz[anchor=base]{
       \path (0,0) node (1) {$\scriptstyle A$} +(2,0) node (2) {$\scriptstyle A\otimes A$} +(0,-2) node (3) {$\scriptstyle A\otimes A$} +(2,-2) node (4) {$\scriptstyle A\otimes A\otimes A$};
       \draw[->] (1) -- (2) node[pos=.5,auto] {$\scriptstyle \Delta$};
       \draw[->] (1) -- (3) node[pos=.5,auto,swap] {$\scriptstyle \Delta$};
       \draw[->] (2) -- (4) node[pos=.5,auto] {$\scriptstyle \Delta\otimes\id$};
       \draw[->] (3) -- (4) node[pos=.5,auto] {$\scriptstyle \id\otimes\Delta$};
     }
   \]commutes;
 \item The subspaces
   \[
     \text{linear span}\{(a\otimes 1)\Delta(b)\ |\ a,b\in A\}
   \]
   and
   \[
     \text{linear span}\{(1\otimes a)\Delta(b)\ |\ a,b\in A\}
   \]
   are norm (resp. $W^*$) dense in $A\otimes A$. 
\end{enumerate} 
\end{definition}

We will often abuse terminology by just writing $A$ for $(A,\Delta)$. 

Note that the comultiplication $\Delta$ makes $A^*$ into a (possibly non-unital) associative ring: If $\varphi$ and $\psi$ are functionals on $A$, then $\varphi\psi$ is the composition
\begin{center}
\begin{tikzpicture}[auto]
  \node (1) at (0,0) {$\scriptstyle A$};
  \node (2) at (2,0) {$\scriptstyle A\otimes A$};
  \node (3) at (4,0) {$\scriptstyle \bC$};
  \draw[->] (1) to node {$\scriptstyle \Delta$} (2);
  \draw[->] (2) to node {$\scriptstyle \varphi\otimes\psi$} (3);
  \draw[->,bend right=30] (1) to node [swap] {$\scriptstyle \varphi\psi$} (3);
\end{tikzpicture}
\end{center}  
It is a state whenever $\varphi$ and $\psi$ are, meaning that $S(A)\subset A^*$ is a multiplicative sub-semigroup.

$C^*$ quantum groups $(A,\Delta)$ come equipped with a distinguished state $h$ playing the same role that the Haar measure does for classical compact groups: For any functional $\varphi\in A^*$, 
\begin{equation}\label{eq.haar}
 h\varphi=\varphi h=h.
\end{equation} 
This is shown in  \cite[4.1]{Wor87} in the separable case and \cite[2.4]{Dae95} in general). The quantum group is \define{reduced} if $h$ is faithful. We can always pass to a reduced version of $A$ by taking the image of the GNS representation of $A$ with respect to $h$. The quantum groups we work with will often be reduced in this sense. 

$W^*$ quantum groups come about in this paper in one way: By starting out with a $C^*$ quantum group, and then taking the $W^*$ closure of the GNS representation corresponding to the Haar state $h$. The resulting von Neumann algebra inherits $h$ as a faithful (normal) functional satisfying the same invariant condition \Cref{eq.haar}.

%%%%%%%%%%%%%%%%%%%%%%%%%%%%%%%%%%%%%%%%%%%%%%%%%%%%%%%%%%%%%%%%%%%%%%%%%%%%%%%%%%%%%%%%%%%%%%%%%%%%%%%%%%%%%%%%%%
\subsection{Actions on spaces}\label{subse.prel_sp}

Quantum groups are interesting because they act on various objects, such as ``quantum spaces''. The latter are possibly non-commutative operator algebras, thought of as algebras of functions on some kind of space. Since passing from spaces to algebras is a contravariant procedure, the relevant definition is as follows:

\begin{definition}\label{def.cqg_act}
Let $(A,\Delta)$ be a $C*$ ($W^*$) quantum group, and $B$ a $C^*$(resp $W^*$)-algebra. A \define{right coaction} of $A$ on $B$ is a $C^*$ (resp. $W^*$) homomorphism $\rho:B\to B\otimes A$ such that
\begin{enumerate}
\renewcommand{\labelenumi}{(\arabic{enumi})}
 \item The diagram
   \[
     \tikz[anchor=base]{
       \path (0,0) node (1) {$\scriptstyle B$} +(2,0) node (2) {$\scriptstyle B\otimes A$} +(0,-2) node (3) {$\scriptstyle B\otimes A$} +(2,-2) node (4) {$\scriptstyle B\otimes A\otimes A$};
       \draw[->] (1) -- (2) node[pos=.5,auto] {$\scriptstyle \rho$};
       \draw[->] (1) -- (3) node[pos=.5,auto,swap] {$\scriptstyle \rho$};
       \draw[->] (2) -- (4) node[pos=.5,auto] {$\scriptstyle \rho\otimes\id$};
       \draw[->] (3) -- (4) node[pos=.5,auto] {$\scriptstyle \id\otimes\Delta$};
     }
   \]commutes;
 \item The subspace
   \[
     \text{linear span}\{(1\otimes a)\rho(b)\ |\ a\in A\ b\in B\}
   \]   
   is dense in $B\otimes A$.
\end{enumerate} 
\end{definition}

We focus mostly on the situation when the group is quantum, but the space is classical, i.e. in the above definition $B$ is either the $C^*$-algebra $\cC(X)$ of all complex-valued continuous functions on a compact Hausdorff space $X$, or its enveloping von Neumann algebra $\cW(X)$ \cite[III.2.4]{Tak02}. 

If $A$ were the algebra $\cC(G)$ of continuous functions on the compact group $G$, then a right coaction of $A$ on $\cC(X)$ would precisely correspond to a \define{left} action of $G$ on $X$ in the following way: If the action $G\times X\to X$ is written as $(x,g)\mapsto xg$, then the corresponding coaction  
\[
 \rho:\cC(X)\to\cC(X)\otimes\cC(G)\cong\cC(X\times G)
\] 
is defined by
\[
 \rho(f)(x,g) = f(xg). 
\]
In other words, coactions are the correct notion of action upon dualizing by passing from spaces to algebras. This justifies the following.

\begin{definition}\label{def.cqg_act_bis}
Let $(A,\Delta)$ be a $C^*$ (resp. $W^*$) quantum group and $X$ a compact Hausdorff topological space. A (right) action of $A$ on $X$ is a coaction of $A$ on $\cC(X)$ (resp. $\cW(X)$) as in \Cref{def.cqg_act}. 
\end{definition}

Now let $\rho$ be an action of $A$ on $X$ as in \Cref{def.cqg_act_bis}, and let $\varphi\in A^*$ (to fix ideas, say we are in the $C^*$ case). The composition
\begin{center}
\begin{tikzpicture}[auto]
  \node (1) at (0,0) {$\scriptstyle \cC(X)$};
  \node (2) at (2,0) {$\scriptstyle \cC(X)\otimes A$};
  \node (3) at (4,0) {$\scriptstyle \cC(X)$};
  \draw[->] (1) to node {$\scriptstyle \rho$} (2);
  \draw[->] (2) to node {$\scriptstyle \id\otimes\varphi$} (3);
\end{tikzpicture}
\end{center}  
 is a self-map on $\cC(X)$, which we denote by $\varphi\triangleright$. As $\varphi$ ranges over $A^*$, these maps constitute a left action of the ring $A^*$ on the vector space $\cC(X)$. Everything works out verbatim in the $W^*$ case provided we substitute $\cW(X)$ for $\cC(X)$. The restriction of this action to the sub-semigroup $S(A)\subset A^*$ will again be denoted by $\triangleright$.  
 
Similarly, the ring $A^*$ acts on the right on the Banach space $\cM(X)$ of complex measures on $X$ (which can be identified with both $\cC(X)^*$ and $\cW(X)^*$). In the $C^*$ setup, for instance, the result $\mu\triangleleft\varphi$ of acting with $\varphi\in A^*$ on $\mu\in\cM(X)$ is a measure defined by   
\begin{center}
\begin{tikzpicture}[auto]
  \node (1) at (0,0) {$\scriptstyle \cC(X)$};
  \node (2) at (2,0) {$\scriptstyle \cC(X)\otimes A$};
  \node (3) at (4,0) {$\scriptstyle \bC$};
  \draw[->] (1) to node {$\scriptstyle \rho$} (2);
  \draw[->] (2) to node {$\scriptstyle \mu\otimes\varphi$} (3);
  \draw[->,bend right=30] (1) to node [swap] {$\scriptstyle\mu\triangleleft\varphi$} (3);
\end{tikzpicture}
\end{center}  
$\triangleleft$ restricts to an action (denoted by the same symbol) of $S(A)$ on $\pr(X)=S(\cC(X))=S(\cW(X))$.

Note that both $\triangleleft$ and $\triangleright$ are affine on $S(A)$ in the sense that 
\[
 (\lambda\psi_1+(1-\lambda)\psi_2)\triangleright f = \lambda(\psi_1\triangleright f)+(1-\lambda)(\psi_2\triangleright f),\ \forall\lambda\in[0,1],\ \forall \psi_i\in S(A) 
\]
and similarly for $\triangleleft$. % Additionally, every self-map $\psi\triangleright$ of $\cC(X)$ or $\cW(X)$ is linear, and every $\triangleleft\psi$ is affine on $\pr(X)$. 

We also need to know what it means for an action to be faithful.

\begin{definition}\label{def.faithful}
A coaction $\rho:B\to B\otimes A$ of a $C^*$ or $W^*$ quantum group $A$ on a $C^*$ or $W^*$-algebra $B$ is \define{faithful} if the set
\[
 \{(\mu\otimes\id)(\rho(b))\ |\ b\in B,\ \mu\in S(B)\}\subset A
\]
generates $A$ as a $C^*$(resp. $W^*$)-algebra.  
\end{definition}

$C^*$ quantum groups acting faithfully on a classical compact space have one property that will be crucial below: They admit an antipode, i.e. a bounded, multiplication-reversing map $\kappa:A\to A$ that plays the same role as the map $f\mapsto f(\bullet^{-1})$ for continuous functions $f$ on an ordinary compact group.  

More precisely: 

Every $C^*$ quantum group $A$ has a unique dense unital $*$-subalgebra $\cA$ with the property that $\Delta(\cA)$ is contained in the algebraic tensor product $\cA\otimes\cA$ (see e.g. \cite[3.1.7]{KusTus99}). The space $\mathrm{End}(\cA)$ of all linear self-maps on $\cA$ is then an associative algebra, with multiplication $(f,g)\mapsto fg$ defined by
\begin{center}
\begin{tikzpicture}[auto]
  \node (1) at (0,0) {$\scriptstyle \cA$};
  \node (2) at (2,0) {$\scriptstyle \cA\otimes\cA$};
  \node (3) at (4,0) {$\scriptstyle \cA\otimes\cA$};
  \node (4) at (6,0) {$\scriptstyle \cA$};
  \draw[->] (1) to node {$\scriptstyle \Delta$} (2);
  \draw[->] (2) to node {$\scriptstyle f\otimes g$} (3);
  \draw[->] (3) to node {$\scriptstyle \text{multiply}$} (4);
  \draw[->,bend right=30] (1) to node [swap] {$\scriptstyle fg$} (4);
\end{tikzpicture}
\end{center}

There is a $*$-algebra homomorphism $\varepsilon:\cA\to\bC$ such that the map $\cA\to\cA$ defined by $a\mapsto\varepsilon(a)1$ is a unit for the algebra structure on $\mathrm{End}(\cA)$ introduced above. Moreover, there is a (left and right) inverse $\kappa\in\mathrm{End}(\cA)$ to the identity map $\id_{\cA}$. In standard terminology, $\cA$ is a \define{Hopf $*$-algebra} with \define{counit} $\varepsilon$ and \define{antipode} $\kappa$. 

In general, $\kappa$ does not extend continuously to a self-map of the $C^*$-algebra $A$; if it does, $(A,\Delta)$ is said to be \define{of Kac type}. As a consequence of \cite[3.23]{Hua12} $A$ is of Kac type if it acts faithfully on a compact space. The extension, denoted by $\kappa$ again, is involutive, reverses multiplication and preserves the $*$ operation.

For a Kac type $C^*$ quantum group, the antipode $\kappa$ extends to a normal (involutive, multiplication-reversing) self-map of the corresponding $W^*$ quantum group obtained via the GNS construction for the Haar state. Although we do use $W^*$ for convenience in \Cref{se.main}, the ones we use are all in a sense ``built out of'' $C^*$ actions via this GNS construction. For this reason, we make

\begin{convention}\label{conv.C*}
Unless specified otherwise, quantum groups and actions are assumed to be $C^*$. 
\end{convention}

%%%%%%%%%%%%%%%%%%%%%%%%%%%%%%%%%%%%%%%%%%%%%%%%%%%%%%%%%%%%%%%%%%%%%%%%%%%%%%%%%%%%%%%%%%%%%%%%%%%%%%%%%%%%%%%%%%
\subsection{Actions on metric spaces}\label{subse.prel_metric}

We now focus on quantum actions on compact spaces $X$ whose topology is induced by a metric $d$. We recall here the two ways in which the action might be compatible with $d$ referred to in the previous section. 

Throughout, $\rho$ stands for an action by a quantum group $(A,\Delta)$ on $X$. We always assume that the action is faithful, and hence $A$ admits an antipode $\kappa$. 

One way to phrase the condition that an action of a compact group $G$ on $X$ be isometric is
\begin{equation}\label{eq.classical_isometry}
 d(xg,y) = d(x,yg^{-1}),\ \forall x,y\in X,\ \forall g\in G. 
\end{equation}
Quantum groups have no points $g$, but we are still acting on a classical space $X$, so we can still phrase the condition in terms of points $x$ and $y$ of $X$. 

Denote by $d_x\in\cC(X)$ the function defined by $d_x(y) = d(x,y)$. A moment's thought will convince the reader that the analogue of \Cref{eq.classical_isometry} is \Cref{eq.D_isometric} below.

\begin{definition}\label{def.D_isometric}
The action $\rho$ is \define{(D)-isometric} or \define{satisfies condition (D)} if 
\begin{equation}\label{eq.D_isometric}
 \rho(d_y)(x) = \kappa(\rho(d_x)(y)),\ \forall x,y\in X, 
\end{equation}
where in general, for $f\in \cC(X)$, $\rho(f)(x)$ is the evaluation at $x$ on the left hand tensorand of $\rho(f)\in\cC(X)\otimes A$ (and analogously for $f\in\cW(X)$). 
\end{definition}

This is \cite[Definition 3.1]{Gos12}, where (D)-isometries are simply called isometries. Goswami shows in \cite[3.5]{Gos12} that under certain conditions on $\rho$, this notion of isometry can be recast in terms of a ``diagonal action'' by $A$ on $X\times X$. We will not need this interpretation here though. 

The other notion of quantum isometry that we review here is that of \cite{QuaSab12}. It is based on the observation that the distance $d$, which induces the Lipschitz seminorm
\[
 L(f) = \sup_{x\ne y}\frac{|f(x)-f(y)|}{d(x,y)}
\] 
on $\cC(X)$, can in turn be recovered from $L$ by
\begin{equation}\label{eq.d_from_L}
 d(x,y) = \sup_{L(f)\le 1}(f(x)-f(y)).  
\end{equation}
Hence, compatibility of $\rho$ with $d$ should be amenable to phrasing in terms of compatibility with $L$ (which we henceforth also refer to as the Lip-norm). This is done as follows.

\begin{definition}\label{def.Lip_isometric}
The action $\rho$ is \define{(Lip)-isometric} or \define{satisfies condition (Lip)} if
\begin{equation}\label{eq.Lip_isometric}
L(\psi\triangleright f)\le L(f),\ \forall \psi\in S(A). 
\end{equation}
for all $f\in\cC(X)$. 
\end{definition}

(Lip)-isometric actions are called isometric in \cite[Definition 3.1]{QuaSab12}. Note that (Lip) is indeed equivalent to $\rho$ being isometric in the usual sense when it is an action of an ordinary compact group \cite[3.4]{QuaSab12}. 

We can give an alternative formulation of condition (Lip). To this end, first regard each $x\in X$ as a probability measure concentrated at a single point. This embeds $X$ in $\pr(X)$ continuously with respect to the given, compact topology on $X$ and the weak$^*$ topology of $\pr(X)$ induced by duality with $\cC(X)$. The distance $d$ can be extended to $\pr(X)$ in the following way:
\begin{equation}\label{eq.KanRub}
 d(\mu,\nu) = \sup_{L(f)\le 1}(\mu(f)-\nu(f)),\ \forall\mu,\nu\in\pr(X)
\end{equation}
(where $\mu(f)=\int_Xf\ \text{d}\mu$); cf. \Cref{eq.d_from_L}. This is the so-called Kantorovi\u c distance on $\pr(X)$ \cite{Kan42,KanRub57}. 

Recall from \Cref{subse.prel_sp} that $\rho$ induces a right action of the semigroup $S(A)$ on $\pr(X)$, denoted by $(\mu,\psi)\mapsto\mu\triangleleft\psi$, $\mu\in\pr(X)$, $\psi\in S(A)$. The following definition anticipates \Cref{subse.prel_Was} below.

\begin{definition}\label{def.Lip1_isometric}
The action $\rho$ is \define{(Lip$_1$)-isometric} or \define{satisfies condition (Lip$_1$)} if
\begin{equation}\label{eq.Lip1_isometric}
d(\mu\triangleleft\psi,\nu\triangleleft\psi)\le d(\mu,\nu),\ \forall \mu,\nu\in\pr(X),\ \forall \psi\in S(A) 
\end{equation}
for all $f\in\cC(X)$.
\end{definition}

In other words, the action of $S(A)$ on $\pr(X)$ is by contractions with respect to $d$. We now have

\begin{proposition}\label{pr.Lip=Lip1}
The following conditions on the action $\rho$ are equivalent:
\begin{enumerate}
\renewcommand{\labelenumi}{(\arabic{enumi})}
 \item (Lip);
 \item (Lip$_1$);
 \item For all $x,y\in X$ and all $\psi\in S(A)$ we have $d(x\triangleleft \psi,y\triangleleft\psi)\le d(x,y)$.
\end{enumerate}
\end{proposition} 
\begin{proof} 
(1) $\Rightarrow$ (2): Let $\mu,\nu$ be probability measures on $X$, and $\psi\in S(A)$ a state on the underlying algebra of the quantum group.

By definition, $d(\mu\triangleleft\psi,\nu\triangleleft\psi)$ is the supremum over $f$ with $L(f)\le 1$ of 
\[
 (\mu\triangleleft\psi)(f)-(\nu\triangleleft\psi)(f) = \mu(\psi\triangleright f)-\nu(\psi\triangleright f). 
\]
(Lip) implies that $\mu\triangleright f$ and $\nu\triangleright(f)$ both have Lip-norm $\le 1$, which means that the quantity cannot be larger than $d(\mu,\nu)$.   

(2) $\Rightarrow$ (3) is obvious. 

(3) $\Rightarrow$ (1): In general, we have 
\[
 (\psi\triangleright f)(x)-(\psi\triangleright f)(y) = f(x\triangleleft\psi) - f(y\triangleleft\psi). 
\]
If $L(f)\le 1$, then this difference is at most $d(x\triangleleft\psi,y\triangleleft\psi)$. In turn, if $\rho$ is (Lip$_1$), then this distance is at most $d(x,y)$. 
\end{proof}

%%%%%%%%%%%%%%%%%%%%%%%%%%%%%%%%%%%%%%%%%%%%%%%%%%%%%%%%%%%%%%%%%%%%%%%%%%%%%%%%%%%%%%%%%%%%%%%%%%%%%%%%%%%%%%%%%%
\subsection{Wasserstein distances}\label{subse.prel_Was}

The notation (Lip$_1$) from the previous subsection is meant to suggest that that condition is an ``$L^1$ version'' of something that more generally can come in an $L^p$ flavor for any $p\ge 1$. We unpack this here. 

As before, $X$ is a compact Hausdorff topological space. We denote by $p_i$, $i=1,2$ the projections $X\times X\to X$ on the first and second factor respectively.

\begin{definition}
Let $\mu,\nu\in\pr(X)$ be two probability measures. A probability measure $\pi\in\pr(X)$ is a \define{$(\mu,\nu)$-coupling} if $\pi_{1*}(\pi)=\mu$ and $\pi_{2*}(\pi)=\nu$. 

We denote the set of $(\mu,\nu)$-couplings by $\Pi(\mu,\nu)$. 
\end{definition}
We also say that $\pi\in\Pi(\mu,\nu)$ has \define{marginals} $\mu$ and $\nu$. 

Recall the following definition from the Introduction:

\begin{definition}\label{def.Wp}
Let $\mu,\nu$ be probability measures on the compact metric space $(X,d)$ and $p\ge 1$. The \define{Wasserstein p-distance} is
\[
 W_p(\mu,\nu) = \inf_{\pi\in\Pi(\mu,\nu)}\left(\int_{X\times X} d(x,y)^p\ \text{d}\pi\right)^{\frac 1p}. 
\]  
\end{definition}

The term `Wasserstein (or Vasershtein) distance' originates in \cite{Dob70}, based on the introduction of $W_1$ in \cite{Vas69}. It can be shown that $W_p$ is indeed a distance (i.e. it satisfies the triangle inequality and only vanishes when $\mu=\nu$), and that moreover it induces on $\pr(X)$ the usual weak$^*$ topology. We refer the reader to \cite[Chapter 6]{Vil09} for further details. 

The connection to the previous discussion is by means of $W_1$. It is a result due to Kantorovich and Rubisntein \cite{KanRub58} that it coincides with the distance $d$ on $\pr(X)$ defined by \Cref{eq.KanRub}. 

\begin{definition}\label{def.Lipp_isometric}
Let $p\ge 1$. An action of a quantum group on the compact metric space $X$ is \define{(Lip$_p$)-isometric} or \define{satisfies condition (Lip$_p$)} if
\begin{equation}\label{eq.Lipp_isometric}
W_p(\mu\triangleleft\psi,\nu\triangleleft\psi)\le W_p(\mu,\nu),\ \forall \mu,\nu\in\pr(X),\ \forall \psi\in S(A) 
\end{equation}
for all $f\in\cC(X)$.
\end{definition}

\begin{remark}
It might seem strange that an isometry requirement is phrased as an inequality, but observe that when the quantum group is simply the algebra of continuous functions on a compact group, the definition is equivalent to the usual notion of isometry:

Compact groups have points $g$, which can be regarded as measures. Condition \Cref{eq.Lipp_isometric} demands that both $\triangleleft g$ and its inverse $\triangleleft g^{-1}$ contract Wasserstein distances, which means that both preserve $W_p$. 
\end{remark}

We now have the following partial analogue of \Cref{pr.Lip=Lip1}.

\begin{proposition}\label{pr.Lipp_alt}
For any $p\ge 1$, condition (Lip$_p$) is equivalent to the requirement that for all $x,y\in X$ and all $\psi\in S(A)$ we have $W_p(x\triangleleft \psi,y\triangleleft\psi)\le d(x,y)$.
\end{proposition}
\begin{proof}
(Lip$_p$) is clearly stronger than the alternative condition in the statement (since the latter is simply (Lip$_p$) applied to measures supported at single points). So we are left with the opposite implication. 

As a consequence of the Kantorovich duality theorem \cite[5.10]{Vil09}, for any pair $\mu,\nu\in\pr(X)$ we have
\[
 W_p^p(\mu,\nu)=\inf_{\pi\in\Pi(\mu,\nu)}\int_{X\times X}d(x,y)^p\ \text{d}\pi = \sup\left(\int_X f\ \text{d}\mu+\int_X g\ \text{d}\nu\right),
\]
where the supremum is taken over all pairs  of continuous functions $f,g\in\cC(X)$ with $f(x)+g(y)\le d(x,y)^p,\ \forall x,y$. Denote this set of pairs $(f,g)$ by $P$. Applying the duality theorem to $\mu\triangleleft\psi$ and $\nu\triangleleft\psi$ for some $\psi\in S(A)$, we find 
\[
 W_p^p(\mu\triangleleft\psi,\nu\triangleleft\psi) = \sup_{(f,g)\in P}\left((\mu\triangleleft\psi)(f)+(\nu\triangleleft\psi)(g)\right).
\]
In turn, this is $\sup(\mu(\psi\triangleright f)+\nu(\psi\triangleright g))$. It will be at most
\[
 \sup_{(f',g')\in P}(\mu(f')+\nu(g')) = W_p^p(\mu,\nu),
\]   
as desired, if we show that $(\psi\triangleright f,\psi\triangleright g)$ is in $P$ whenever $(f,g)$ is. But for $(f,g)\in P$ we have
\begin{multline*}
 (\psi\triangleright f)(x) + (\psi\triangleright g)(y) = (x\triangleleft\psi)(f) + (y\triangleleft\psi)(g) \\ 
\le \sup_{(f',g')\in P}((x\triangleleft\psi)(f') + (y\triangleleft\psi)(g')) = W_p^p(x\triangleleft\psi,y\triangleleft\psi). 
\end{multline*}
Since we are assuming that $\rho$ satisfies the condition from the statement of the proposition, this is at most $W_p^p(x,y)=d(x,y)^p$. 
\end{proof}

This implies that the conditions (Lip$_p$) are totally ordered by strength:

\begin{corollary}\label{cor.Lipp_implications}
For any quantum action on $X$ and any $1\le q\le p$, (Lip$_p$) implies (Lip$_q$). 
\end{corollary}
\begin{proof}
Suppose the action is (Lip$_p$)-isometric. Then, for $x,y\in X$ and $\psi\in S(A)$ we have
\[
 W_q(x\triangleleft\psi,y,\triangleleft\psi)\le W_p(x\triangleleft\psi,y\triangleleft\psi)\le d(x,y),
\]
where the first inequality is simply $W_q\le W_p$ (a consequence of H\"older's inequality which we leave to the reader) and the second one is (Lip$_p$). The conclusion follows from \Cref{pr.Lipp_alt}. 
\end{proof}

%%%%%%%%%%%%%%%%%%%%%%%%%%%%%%%%%%%%%%%%%%%%%%%%%%%%%%%%%%%%%%%%%%%%%%%%%%%%%%%%%%%%%%%%%%%%%%%%%%%%%%%%%%%%%%%%%%
%%%%%%%%%%%%%%%%%%%%%%%%%%%%%%%%%%%%%%%%%%%%%%%%%%%%%%%%%%%%%%%%%%%%%%%%%%%%%%%%%%%%%%%%%%%%%%%%%%%%%%%%%%%%%%%%%%
\section{A continuous Hall theorem}\label{se.Hall}

The main result of this section will be used in the course of the proof of \Cref{th.main}, but it might be of some independent interest. Before stating it, we need some preparations. 

As before, $X$ is a metrizable compact Hausdorff topological space and the map $p_1$ (respectively $p_2$) from $X\times X$ to $ X$ is the projection on the left (resp. right) hand factor. For subsets $Y\subseteq X\times X$ and $S\subseteq X$ we write 
\[
 p_{12}^Y(S) = p_2\left(p_1^{-1}(S)\cap Y\right)
\]  
and 
\[
 p_{21}^Y(S) = p_1\left(p_2^{-1}(S)\cap Y\right)
\]
respectively. In other words, $p_{12}^Y(S)$ (resp. $p_{21}^Y(S)$) is the set of points $x'\in X$ such that $(x,x')\in Y$ (resp. $(x',x)\in Y$) for some $x\in S$. 

All is in place now for the statement.

\begin{theorem}\label{th.Hall}
Let $X$ be a compact Hausdorff topological space, $Y\subseteq X\times X$ a closed subset, and $\mu,\nu\in\pr(X)$ probability measures. 

Then, there is a coupling $\pi\in\Pi(\mu,\nu)$ supported on $Y$ if and only if
\begin{equation}\label{eq.Hall}
 \nu\left(p_{12}^Y(S)\right)\ge\mu(S)\ \text{for every closed subset}\ S\subseteq X. 
\end{equation}
\end{theorem}

Before going into the proof, a word on the title of this section. It refers to Hall's theorem, a combinatorial result on matchings in bipartite graphs. 

Recall that a graph is \define{bipartite} if its vertex set can be partitioned in two classes $A$ and $B$ such that every edge connects a vertex in $A$ and one in $B$. A \define{perfect matching} is a set of edges such that every vertex is on precisely one. In general, for a vertex set $S$, denote by $N(S)$ the set of neighbors of $S$, i.e. vertices connected to some element of $S$. 

The result (also known as Hall's marriage theorem because it pairs off two classes of vertices) reads

\begin{corollary}\label{cor.classical_Hall}
A graph with bipartition into sets $A,B$ of equal size has a perfect matching if and only if 
\begin{equation}\label{eq.classical_Hall}
 |N(S)| \ge |S|,\ \forall S\subseteq A.  
\end{equation}
\end{corollary}

See e.g. \cite[2.1.2]{Die00}, where the theorem is stated in somewhat more generality (and perfect matchings are called 1-factors). We do not need it in the sequel, but present a short proof of it as a consequence of \Cref{th.Hall} to illustrate the analogy.

\begin{proof}
Since $A$ and $B$ have the same size ($n$, say), we can identify them to some common $n$-element set $X$. Denote by $Y\subseteq A\times B\cong X\times X$ the set of edges of the graph. 

Taking $\mu=\nu$ to be the normalized counting measure on $X$, the conclusion follows from \Cref{th.Hall}: Condition \Cref{eq.Hall} in this particular case is exactly \Cref{eq.classical_Hall} divided by $n$.   
\end{proof}

As mentioned in the Introduction, the equivalence of (D) and (Lip) when $X$ is finite is proven in \cite[Theorem 3.5]{QuaSab12}. The implication (D) $\Rightarrow$ (Lip) follows a path we replicate here, and uses a matching-type result that is a a kind of hybrid between \Cref{th.Hall} and \Cref{cor.classical_Hall} (\cite[Lemma 4.5]{QuaSab12}): $X$ is discrete but $\mu$ and $\nu$ are not necessarily counting measures. It too is a consequence of \Cref{th.Hall}, but we will not expand on this further.  
 
We make some preparations in view of proving \Cref{th.Hall}. Recall that an \define{ordered vector space} $(V,\le)$ is a real vector space $V$ equipped with a partial order $\le$ compatible with the vector space structure in the sense that sums of non-negative elements are non-negative and $rv\ge 0$ for all $v\ge 0$ and $r\in\bR_{\ge 0}$.

\begin{definition}\label{def.ous}
An \define{order unit} for an ordered (real) vector space $(V,\le)$ is an element $0\le e\in V$ such that for every $v\in V$ there is some $r\ge 0$ such that $re\ge v$.

An \define{order unit space} $(V,\le,e)$ is an ordered vector space $(V,\le)$ equipped with an order unit $e$.  
\end{definition}
See e.g \cite[Definition 2.4]{PauTom09}. The term `order unit space' is not used in that paper, but it is prevalent in the literature. 

The only examples of order unit spaces that will come up below are of the form $\cC_\bR(Z)$, the space of all real-valued continuous functions on a compact Hausdorff space $Z$. The ordering is the usual one, and the order unit is the constant function $1$.

\begin{proof of Hall}
($\Rightarrow$) Suppose there is some $(\mu,\nu)$-coupling $\pi$ supported on $Y$, and let $S\subseteq A$ be some closed subset. Then, we have
\begin{multline*}
 \nu\left(p_{12}^Y(S)\right) = \nu\left(p_2\left(p_1^{-1}(S)\cap Y\right)\right) \\
 = \pi\left( p_2^{-1}\left(p_2\left(p_1^{-1}(S)\cap Y\right)\right)\right)\ge \pi\left( p_1^{-1}(S)\cap Y\right)\\
 =\pi\left( p_1^{-1}(S)\right) = \mu(S), 
\end{multline*}
where: 

The second equality follows from $\nu=p_{2*}(\pi)$; the inequality is just the inclusion of any set $\bullet\subseteq X\times X$ in $\left(p_2^{-1}\circ p_2\right)(\bullet)$; the third equality is the assumption that $\pi$ is supported on $Y$, and the last equality is just $\mu=p_{1*}(\pi)$. 

($\Leftarrow$) Consider the vector spaces $V=\cC_\bR(X)^{\oplus 2}$ (real-valued continuous functions on two copies of $X$) and $W=\cC_\bR(Y)$. We regard both as order unit spaces, as indicated in the preceding discussion.  

We have a linear map $T:V\to W$ defined by
\[
 T(f,g)(x,x') = f(x) + g(x'),\ \forall (x,x')\in Y. 
\] 
The restriction of the dual map $T^*$ restricts to $(p_{1*},p_{2*}):\pr(Y)\to\pr(X)^{\oplus 2}$, and we need to show that $(\mu,\nu)$ is in the image of this restriction. 

We attempt to construct $\pi\in\Pi(\mu,\nu)$ as follows. We need to have $\pi(T(f,g)) = \mu(f) + \nu(g)$. Suppose (a) this is well defined, and (b) this value is non-negative whenever $T(f,g)\ge 0$. 

Then, $\pi$ is a positive functional on the ordered subspace $T(V)\le W$, and this subspace contains the order unit of $W$ (since $T(1,0)=1$). By a standard Hahn-Banach-type result for order unit spaces (\cite[2.15]{PauTom09}), $\pi$ can be extended to a positive functional on $W$. Since moreover $\pi(1)=1$, $\pi$ must be a probability measure on $Y$. 

We are left having to prove (a) and (b) above. In fact, (b) implies (a): If $\mu(f)+\mu(g)\ge 0$ whenever $T(f,g)\ge 0$, then, reversing the signs of $f$ and $g$, $\mu(f)+\mu(g)\le 0$ if $T(f,g)\le 0$. Both inequalities hold when $T(f,g)=0$, hence (a). 
The only thing left to show is that if \Cref{eq.Hall} holds, then 
\[
 T(f,g)\ge 0\Rightarrow \mu(f) + \nu(g)\ge 0. 
\] 
This is taken care of by the next lemma. 
\end{proof of Hall}

\begin{lemma}\label{le.Hall_aux}
Let $\mu,\nu$ be two probability measures on a metrizable compact Hausdorff space $X$, and $Y\subseteq X\times X$ a closed subset. If \Cref{eq.Hall} holds, then for $f,g\in\cC(X)$ we have 
\[
 f(x) + g(x') \ge 0,\ \forall (x,x')\in Y\quad \Rightarrow\quad \mu(f)+\nu(g)\ge 0. 
\]
\end{lemma}
\begin{proof}
Adding a constant to $f$ and subtracting it from $g$, we may as well assume $g\le 0$. Note also that the inequality in \Cref{eq.Hall} holds for all \define{locally} closed subsets $S$ of $X$ (i.e. intersections of a closed and an open set). This is because all such sets are $F_\sigma$, i.e. countable unions of closed subsets; this is one place where the condition that $X$ is metrizable comes in.  

Let $\widetilde f$ be a lower semicontinuous step function on $X$ such that $0\le \widetilde f-f\le\varepsilon$ for some small $\varepsilon>0$. Let $a_1,\ldots,a_n$ be the values of $s$. The sets $A_i=\widetilde f^{-1}(a_i)$ are locally closed, and constitute a partition of $X$. We leave it to the reader to show that $B_i=p_{12}^Y(A_i)$ cover $X$ (this is a consequence of \Cref{eq.Hall}). 

Because they are \define{analytic} subsets of $X$, i.e. continuous images of Borel subsets of $X\times X$, $B_i$ are measurable with respect to any Borel measure (e.g. \cite[3.2.4]{Arv76}; this again uses metrizability). Hence, everything below makes sense:
\begin{equation}\label{eq.Hall_aux}
 \mu\left(\widetilde f\right) + \nu(g) = \sum_{i=1}^n\left(\int_{A_i}\widetilde f\ \text{d}\mu\right) +\nu(g) \ge \sum_{i=1}^n\left(\int_{A_i}\widetilde f\ \text{d}\mu+\int_{B_i} g\ \text{d}\nu\right) \ge 0.
\end{equation}
Here, the first inequality follows from $g\le 0$ and $\bigcup_iB_i=X$, while the second one is a consequence of 
\[
 \widetilde f|_{A_i} = a_i \stackrel{\Cref{eq.Hall}}{\Longrightarrow} g|_{B_i}\ge -a_i.
\]
Since the left hand side of \Cref{eq.Hall_aux} is within the arbitrarily small $\varepsilon$ of $\mu(f)+\nu(g)$, this finishes the proof. 
\end{proof}

%%%%%%%%%%%%%%%%%%%%%%%%%%%%%%%%%%%%%%%%%%%%%%%%%%%%%%%%%%%%%%%%%%%%%%%%%%%%%%%%%%%%%%%%%%%%%%%%%%%%%%%%%%%%%%%%%%
%%%%%%%%%%%%%%%%%%%%%%%%%%%%%%%%%%%%%%%%%%%%%%%%%%%%%%%%%%%%%%%%%%%%%%%%%%%%%%%%%%%%%%%%%%%%%%%%%%%%%%%%%%%%%%%%%%
\section{Main results}\label{se.main}

We are now ready to state the main theorem of the paper. We reprise the notations from previous sections: $(X,d)$ is a compact metric space, $(A,\Delta)$ is a $C^*$ quantum group acting on $X$ via $\rho$. 

As before, we associate to $\rho$ the actions $\triangleleft$ and $\triangleright$ of $S(A)$ on $\cC(X)$ or $\cW(X)$ and on $\pr(X)=S(\cC(X))=S(\cW(X))$ respectively. We leave the symbols `$\triangleleft$' and `$\triangleright$' unadorned, as there will be no danger of confusion: They always refer to the actions induced by $\rho$.

\begin{theorem}\label{th.main}
If the quantum action $\rho$ on $X$ is (D)-isometric, then for every $x,y\in X$ and for every state $\psi\in S(A)$ there is a coupling $\pi\in\Pi(x\triangleleft \psi,y\triangleleft\psi)$ supported on
\begin{equation}\label{eq.main}
 \{(x',y')\in X\times X\ |\ d(x',y') = d(x,y)\}. 
\end{equation}
\end{theorem}

Before proving this, we record the following consequence, announced earlier.

\begin{corollary}\label{cor.main}
Conditions (D) and (Lip$_p$), $p\ge 1$ for $\rho$ are totally ordered by strength, with (D) being the strongest and (Lip$_p$) being stronger for larger $p$. 
\end{corollary}
\begin{proof}
Since the relative strength of the various (Lip$_p$) conditions is covered by \Cref{cor.Lipp_implications}, we only need to show that (D) implies all (Lip$_p$). This follows immediately from \Cref{th.main}: 

According to \Cref{pr.Lipp_alt}, it is enough to prove that $W_p^p(x\triangleleft\psi,y\triangleleft\psi)\le d(x,y)^p$ for any choice of $x,y\in X$ and $\psi\in S(A)$. Denote $\Pi=\Pi(x\triangleleft\psi,y\triangleleft\psi)$ for some such choice. We have
\[
W_p^p(x\triangleleft\psi,y\triangleleft\psi)\quad =\quad \inf_{\widetilde{\pi}\in\Pi}\widetilde{\pi}\left(d^p\right)\quad \le\quad \pi(d^p)\quad =\quad d(x,y)^p,
\]
where $\pi$ is a coupling as in the statement of the theorem, concentrated on those pairs of points $d(x,y)$ apart. 
\end{proof}

In preparation for proving \Cref{th.main}, we place ourselves entirely within the $W^*$ setting as follows.

Suppose $\rho:\cC(X)\to \cC(X)\otimes A$ is a $C^*$ quantum action. Let $\overline{A}$ be the reduced $W^*$ quantum group constructed via the GNS representation of $A$ with respect to the Haar state $h$, as explained in \Cref{subse.prel_cqg}. By the universality property of the $W^*$ envelope $\cC(X)\to\cW(X)$, the composition
\begin{center}
\begin{tikzpicture}[auto]
  \node (1) at (0,0) {$\scriptstyle \cC(X)$};
  \node (2) at (2,0) {$\scriptstyle \cC(X)\otimes A$};
  \node (3) at (4.3,0) {$\scriptstyle \cW(X)\otimes\overline{A}$};
  \draw[->] (1) to node {$\scriptstyle \rho$} (2);
  \draw[->] (2) to (3);
  %\draw[->,bend right=30] (1) to node [swap] {$\scriptstyle\mu\triangleleft\varphi$} (3);
\end{tikzpicture}
\end{center}  
factors uniquely through a $W^*$-algebra morphism $\cW(X)\to \cW(X)\otimes\overline{A}$ satisfying all of the requirements of a coaction.  

Consequently, we can (and will, throughout the proof of \Cref{th.main}) assume that $A$ is a reduced $W^*$ obtained by the GNS construction from a $C^*$ quantum group acting faithfully on $X$ in the sense of \Cref{def.cqg_act_bis,def.faithful}.
 
For every Borel subset $S$ of $X$, the characteristic function $\chi_S$ is well defined as an element of $\cW(X)$. For $x\in X$ we denote by $a_{x;S}$ the image of $\chi_S$ through the composition
\begin{center}
\begin{tikzpicture}[auto]
  \node (1) at (0,0) {$\scriptstyle \cW(X)$};
  \node (2) at (2,0) {$\scriptstyle \cW(X)\otimes A$};
  \node (3) at (4,0) {$\scriptstyle A$,};
  \draw[->] (1) to node {$\scriptstyle \rho$} (2);
  \draw[->] (2) to node {$\scriptstyle x\otimes\id$} (3);
  %\draw[->,bend right=30] (1) to node [swap] {$\scriptstyle\mu\triangleleft\varphi$} (3);
\end{tikzpicture}
\end{center}  
where the $x$ over the second arrow stands for the measure concentrated at the point $x$ (so we regard it as an element of $S(A)$).

\begin{remark}\label{re.sigma_alg_hom}
Note that since both arrows in the diagram are von Neumann algebra morphisms and $\chi_S$ is a projection (i.e. a self-adjoint idempotent), so is $a_{x;S}$. 

In addition, for fixed $x\in X$, the map $S\mapsto a_{x;S}$ is a Boolean $\sigma$-algebra homomorphism from the Borel sets of $X$ onto some Boolean $\sigma$-subalgebra of the complete poset of all projections in $\cW(X)$. We use these observations repeatedly below. 
\end{remark}

We first phrase condition (D) in terms of the elements $a_{x;S}$. For $x\in X$ and $r\ge 0$ denote by $B(x,r)\subseteq X$ the closed ball of radius $r$ around $x$.

\begin{lemma}\label{le.D_in_terms_of_a}
The $W^*$ quantum action $\rho$ of $A$ on $X$ is (D)-isometric if and only if we have 
\begin{equation}\label{eq.D_in_terms_of_a}
 a_{x;B(y,r)} = \kappa(a_{y;B(x,r)}),\ \forall x,y\in X,\ \forall r\ge 0.
\end{equation}
\end{lemma}
\begin{proof}
We write $\cW=\cW(X)$. For $f\in\cW\otimes\cW$ denote $f(x,-)\in\cW$ by $f_x$ and $f(-,x)\in W$ by $f^x$ (this matches the notation $d_x=d(x,-)$ from before, and $d^x=d_x$ because $d$ is symmetric). 

Note that for $x,y\in X$ both 
\[
 f\mapsto \rho(f_y)(x)
\] 
and 
\[
 f\mapsto \kappa(\rho(f^x)(y)) 
\]
are von Neumann algebra morphisms from $\cW\otimes\cW$ to $A$. Indeed, the only minor difficulty is that $\kappa$ \define{reverses} multiplication. Nevertheless, here we are restricting $\kappa$ the image of $(y\otimes\id)\circ\rho:\cW\to A$; since $\cW$ is commutative, so is its image.

With this observation in place, it follows that \Cref{eq.D_isometric} holds if and only if 
\begin{equation}\label{eq.D_isometric_gen}
 \rho(f_y)(x) = \kappa(\rho(f^x)(y))
\end{equation}  
holds for all $f$ in the von Neumann subalgebra of $\cW\otimes\cW$ generated by $d$, or equivalently, just for a set of generators of this $W^*$-algebra.

One such set of generators consists of the characteristic functions $\chi_{d^{-1}([0,r])}$, $r\ge 0$, where $d^{-1}([0,r])$ is the set of pairs of points at most $r$ apart. 

All that is left now is to observe that the left hand side of \Cref{eq.D_isometric_gen} applied to $f=\chi_{d^{-1}([0,r])}$ yields precisely $a_{x;B(y,r)}$, and similarly, its right hand side is $\kappa(a_{y;B(x,r)})$. 
\end{proof}

\begin{remark}\label{re.D_in_terms_of_a}
Let $R\subseteq\bR$ be a Borel subset, and denote by $B(x,R)\subseteq X$ the set of points in $X$ whose distance from $x\in X$ belongs to $R$. 

By an almost identical proof, the statement of \Cref{le.D_in_terms_of_a} holds with $B(x,R)$ and $B(y,R)$ substituted for $B(x,r)$ and $B(y,r)$ respectively. 
\end{remark}

\begin{lemma}\label{le.aa=0}
Assume the quantum action $\rho$ is (D)-isometric, and let $x,y\in X$. If $S,T\subseteq X$ are Borel subsets such that for some $\delta>0$ we have
\begin{equation}\label{eq.aa=0}
 |d(s,t)-d(x,y)|\ge\delta,\ \forall s\in S,\ t\in T,
\end{equation}
then $a_{x;S}a_{y;T}=0$. 
\end{lemma}
\begin{proof}
Fix a number $r>0$ such that $2r<\delta$. If we prove the conclusion under the assumption that $S$ is contained in a ball $B(z,r)$ of radius $r$, we can finish the proof of the lemma as follows: 

Partition $S$ into Borel subsets $S_i$ contained in balls of radius $r$. Since $a_{x;S_i}a_{y;T}=0$, the additivity of the map $Y\mapsto a_{x;Y}$ (i.e. disjoint union turns into addition) implies
\[
 a_{x;S}a_{y;T} = \sum_i a_{x;S_i}a_{y;T} = 0. 
\]

We henceforth assume $S\subseteq B(z,r)$ for some $z\in X$. This assumption and \Cref{eq.aa=0} together imply
\begin{equation}\label{eq.aa=0_bis}
 |d(z,t)-d(x,y)|\ge \delta-r,\ \forall t\in T.
\end{equation}
In other words, if $R$ is the interval $[d(x,y)-(\delta-r),d(x,y)+(\delta-r)]$ then, using the notation from \Cref{re.D_in_terms_of_a}, we have $T\subseteq B(z,R)$. 

It follows from the discussion above that the inequalities $a_{x;S}\le a_{x,B(z,r)}$ and $a_{y;T}\le a_{y;B(z,R)}$ hold in the poset of projections in $\cW$. In conclusion, it suffices to show that $a_{x;B(z,r)}a_{y;B(z,R)}$ vanishes. 

In the slightly more general version from \Cref{re.D_in_terms_of_a}, \Cref{le.D_in_terms_of_a} implies the first equality below:
\[
 a_{x;B(z,r)}a_{y;B(z,R)} = S(a_{z;B(x,r)})S(a_{z;B(y,R)}) = S(a_{z;B(y,R)}a_{z;B(x,r)}) = 0.
\]
The second equality is just the multiplication-reversing property of $S$, while the last one uses the fact that $B(x,r)$ and $B(y,R)$ are disjoint (because $2r<\delta$), and so $a_{z;B(x,r)}$ and $a_{z;B(y,R)}$ are orthogonal projections in $\cW$. 
\end{proof}

\begin{proof of main}
Fix $x,y\in X$ and $\psi\in S(A)$. We denote $x\triangleleft\psi$ and $y\triangleleft\psi$ by $\mu$ and $\nu$ respectively, and the set \Cref{eq.main} by $Y\subseteq X\times X$ (that is, the set of pairs of points in $X$ that are $d(x,y)$ apart). 

By the continuous version of Hall's theorem (\Cref{th.Hall}), we have to prove that \Cref{eq.Hall} holds:  
\[
 \nu\left(p_{12}^Y(S)\right)\ge\mu(S)\ \text{for every closed subset}\ S\subseteq X.
\]

Denote $T=p_{12}^Y(S)$. Unpacking the definitions, we have
\begin{equation}\label{eq.mu=psi(a)}
 \mu(S) = \psi(a_{x;S})\ \text{and}\ \nu(T) = \psi(a_{y;T}). 
\end{equation}
We'll get the desired conclusion that the left hand side of \Cref{eq.mu=psi(a)} is no larger than the right hand side if we show that $a_{x;S}\le a_{y;T}$ in $\cW$ or equivalently, that the product $a_{x;S}a_{y;X-T}$ is zero.

The set $X-T = X-p_{12}^Y(S)$ consists of those points in $X$ whose distance to every $s\in S$ is \define{not} $d(x,y)$. Since we've chosen $S$ to be closed, this means that $X-T$ is the union of the sets
\[
 U_n = \{z\in X\ |\ |d(z,s)-d(x,y)|\ge\frac 1n,\ \forall s\in S\}. 
\] 
By \Cref{le.aa=0} we know that $a_{x;S}a_{y;U_n}=0$ for every $n>0$. Since $a_{y;X-T}$ is the limit the $a_{y;U_n}$'s in the weak$^*$ topology on $\cW$, $a_{x;S}a_{y;X-T}=0$ follows from the weak$^*$ continuity of multiplication by a fixed element.  
\end{proof of main}

As mentioned before, all isometry conditions in this paper are known to be equivalent when the metric space $X$ is finite (this is \cite[Theorem 3.5]{QuaSab12}). As a consequence of this, we have

\begin{proposition}\label{pr.main_conv}
If the quantum group $A$ acting on $(X,d)$ is finite-dimensional, then all conditions (Lip$_p$) are equivalent, and equivalent to (D). 
\end{proposition} 
\begin{proof}
Since according to \Cref{cor.main} (Lip$_1$) is the weakest among the conditions and (D) the strongest, we only have to prove that the two are equivalent. 

For every point $x\in X$ the $C^*$-algebra map $\cC(X)\to A$ defined by 
\begin{center}
\begin{tikzpicture}[auto]
  \node (1) at (0,0) {$\scriptstyle \cC(X)$};
  \node (2) at (2,0) {$\scriptstyle \cC(X)\otimes A$};
  \node (3) at (4,0) {$\scriptstyle A$,};
  \draw[->] (1) to node {$\scriptstyle \rho$} (2);
  \draw[->] (2) to node {$\scriptstyle x\otimes\id$} (3);
  %\draw[->,bend right=30] (1) to node [swap] {$\scriptstyle\mu\triangleleft\varphi$} (3);
\end{tikzpicture}
\end{center}
(where $\rho$ is the action) has a finite-dimensional commutative image. This image can be identified with $\cC(O_x)$ for a finite subspace $O_x\subset X$, the \define{orbit} of $x$ under the action. 

The closed subspaces $O_x$ can always be defined as we just have, but the finite-dimensionality of $A$ implies that orbits partition $X$. 

Indeed, when $A$ is finite-dimensional it is a Hopf algebra, as explained in \Cref{subse.prel_sp}. We can now apply standard results on Hopf algebra coactions on spaces; the claim that the $O_x$ partition $X$, for instance, follows from \cite[Theorem 3.3]{Skr04} (Skryabin writes $O(\mathfrak{p})$ for prime ideals $\mathfrak{p}$ in the algebra coacted upon by the Hopf algebra; we apply his results to the maximal ideals of $\cC(X)$, which are simply the points $x\in X$). 
We leave it to the reader to check that an orbit $O=O_x$ is invariant under $\rho$, in the sense that there is a factorization
\begin{center}
\begin{tikzpicture}[auto,baseline=(current  bounding  box.center)]
  \node (1) at (0,0) {$\scriptstyle \cC(X)$};
  \node (2) at (2,0) {$\scriptstyle \cC(X)\otimes A$};
  \node (3) at (4,0) {$\scriptstyle \cC(O)\otimes A$;};
  \node (4) at (2,-1) {$\scriptstyle \cC(O)$};
  \draw[->] (1) to node {$\scriptstyle \rho$} (2);
  \draw[->] (2) to node {} (3);
  \draw[->,bend right=30] (1) to node [swap] {} (4);
  \draw[->,bend right=30] (4) to node [swap] {} (3);
  %\draw[->,bend right=30] (1) to node [swap] {$\scriptstyle\mu\triangleleft\varphi$} (3);
\end{tikzpicture}
\end{center}
of the horizontal composition through an action of $A$ on $\cC(O)$; the same goes for any finite union of orbits. This means that every finite subset of $X$ is contained in a $\rho$-invariant finite subset. Hence, when verifying either the (D)-isometry condition \Cref{eq.D_isometric} or the (Lip$_1$) condition as in (3) of \Cref{pr.Lip=Lip1}, we can restrict our attention to an action of $A$ on a finite subspace of $X$. But for such actions we know that the two conditions are equivalent from \cite{QuaSab12}. 
\end{proof}

It is natural now to conjecture that (Lip$_1$) implies (D). We propose the following version, that seems more tractable: The \define{conjunction} of all (Lip$_p$) implies (D). Let us rephrase this somewhat.

\begin{lemma}\label{le.all_Lipp}
Let $(X,d)$ be a compact metric space space, $\mu,\nu\in\pr(X)$ two probability measures, and $r\ge 0$. The following conditions are equivalent:
\begin{enumerate}
\renewcommand{\labelenumi}{(\arabic{enumi})}
 \item $W_p(\mu,\nu)\le r,\ \forall p\ge 1$;
 \item There is a coupling $\pi\in\Pi(\mu,\nu)$ supported on the set of pairs of points $\le r$ apart.  
\end{enumerate}
\end{lemma}
\begin{proof}
That (2) implies (1) is clear from \Cref{def.Wp}. 

Conversely, suppose (1) holds. For every $p\ge 1$, let $\pi_p$ be a $(\mu,\nu)$-coupling such that 
\begin{equation}\label{eq.all_Lipp}
 \pi_p(d^p)^{\frac 1p}\le r+\frac 1p.
\end{equation} 
For every $\varepsilon>0$, no matter how small, the $\pi_p$-measure of the set 
\[
 S_\varepsilon = \{(z,z')\in X\times X\ |\ d(z,z')\ge r+\varepsilon\} 
\]  
goes to zero as $p$ grows to infinity (or \Cref{eq.all_Lipp} would be contradicted). This means that any limit point in $\pr(X\times X)$ of $\pi_p$ as $p\to\infty$ will be supported on the complement of all $S_\varepsilon$. Since $\Pi(\mu,\nu)$ is closed in $\pr(X\times X)$, any such limit point meets the requirements of condition (2) from the statement.
\end{proof}

Consequently, the conjunction of all conditions (Lip$_p$) for a quantum action can be phrased in terms of supports of couplings as in the following conjecture, which seems  easier to handle than (Lip$1$) $\Rightarrow$ (D):

\begin{conjecture}
A faithful quantum action of $A$ on the compact metric space $(X,d)$ is (D)-isometric if for any $x,y\in X$ and $\psi\in S(A)$ there is a $(x\triangleleft\psi,y\triangleleft\psi)$-coupling supported on 
\[
 \{(z,z')\in X\times X\ |\ d(z,z')\le d(x,y)\}. 
\]
\end{conjecture}

We end this section with the following observation. It follows from \Cref{le.D_in_terms_of_a}, and it addresses an issue raised e.g. in \cite[3.6]{Gos12} as well as earlier in \cite{Wan98}: whether or not the structure map $\rho:\cC(X)\to \cC(X)\otimes A$ of a coaction is one-to-one. We show that this is indeed the case for (D)-isometric coactions.

\begin{proposition}\label{pr.injectivity} 
If the quantum action $\rho$ is (D)-isometric, then it is one-to-one. 
\end{proposition}
\begin{proof}
We will make use of the associated $W^*$ action, as explained in preparation for the proof of \Cref{th.main}. 

Suppose $\rho:\cC(X)\to\cC(X)\otimes A$ is not one-to-one. It must then factor through $\cC(Z)$ for some proper closed subspace $Z\subset X$. In particular, for every $x\in X$ and $\psi\in S(A)$ the measure $x\triangleleft\psi$ is supported on $Z$. 

Take $x\in X-Z$, and let $\psi\in S(A)$ be an extension to the entire $A$ of a \define{pure} state on the commutative $C^*$-algebra $((x\otimes\id)\circ\rho)(\cC(X))$. The purity assumption implies that $x\triangleleft\psi$ is multiplicative, and hence pure on $X$; in other words, it is a point $z$, by necessity belonging to $Z$. 

It is easy to see that $x\triangleleft\psi=z$ is equivalent to $\psi(a_{x;z})=1$, where $a_{x;z}$ is shorthand for what above we would have called $a_{x;\{z\}}$. By \Cref{le.D_in_terms_of_a} (with $r=0$), we also have $(\psi\circ\kappa)(a_{z;x})=1$. In turn, this means $z\triangleleft(\psi\circ\kappa) = x\not\in Z$, contradicting the fact that every measure of the form $(\text{measure on }X)\triangleleft(\text{state on }A)$ is supported on $Z$.   
\end{proof}

%%%%%%%%%%%%%%%%%%%%%%%%%%%%%%%%%%%%%%%%%%%%%%%%%%%%%%%%%%%%%%%%%%%%%%%%%%%%%%%%%%%%%%%%%%%%%%%%%%%%%%%%%%%%%%%%%%
%%%%%%%%%%%%%%%%%%%%%%%%%%%%%%%%%%%%%%%%%%%%%%%%%%%%%%%%%%%%%%%%%%%%%%%%%%%%%%%%%%%%%%%%%%%%%%%%%%%%%%%%%%%%%%%%%%
\section{Largest quantum subgroup acting isometrically}\label{se.sbgp}

Classically, whenever a compact group acts on a compact metric space $(X,d)$, there is a largest (closed) subgroup whose action preserves $d$. This section is devoted to replicating this situation for (D)-isometric quantum actions (and partially for (Lip$_p$)-isometric ones). 

The compact metric space $(X,d)$ is fixed throughout. It will be convenient to make quantum actions on it in the sense of \Cref{def.cqg_act_bis} into a category as follows.

\begin{definition}\label{def.actions_category}
If $\rho_A$ and $\rho_B$ are $C^*$ quantum actions of $(A,\Delta_A)$ and $(B,\Delta_B)$ respectively on $X$ a \define{morphism} $\rho_A\to\rho_B$ is a $C^*$-algebra map $T:A\to B$ intertwining $\Delta_A$ and $\Delta_B$ and making the triangle 
\[
  \tikz[anchor=base]{
   \path (0,0) node (1) {$\scriptstyle \cC(X)$} +(3,1) node (2) {$\scriptstyle \cC(X)\otimes A$} +(3,-1) node (3) {$\scriptstyle \cC(X)\otimes B$};
   \draw[->] (1) -- (2) node[pos=.5,auto] {$\scriptstyle \rho_A$};
   \draw[->] (1) -- (3) node[pos=.5,auto,swap] {$\scriptstyle \rho_B$};
   \draw[->] (2) -- (3) node[pos=.5,auto] {$\scriptstyle \id\otimes T$};
  }
\]commute. 

We denote the category of $C^*$ actions on $X$ by $\cact$.  
\end{definition}

We usually label the morphism $\rho_A\to\rho_B$ by the same symbol as its underlying algebra map $A\to B$. 

Consider an action $\rho_A$ of a quantum group $(A,\Delta_A)$ on the compact metric space $(X,d)$. We seek a solution to the following universal problem:

\begin{definition}\label{def.iso_env}
A (D)-\define{isometric envelope} of $\rho_A$ is a (D)-isometric action $\rho_B$ on $(X,d)$ by $(B,\Delta_B)$ equipped with a morphism $T:\rho_A\to\rho_B$ in $\cact$ satisfying the following property: 

Every morphism $\rho_A\to\rho_C$ from $\rho_A$ into a (D)-isometric action factors uniquely through $T$.  
\end{definition}

\begin{remark}
As is usually the case with universal objects, if a (D)-isometric envelope exists it is unique up to unique isomorphism. 

Moreover, it is an easy exercise to show that the map $A\to B$ underlying a (D)-isometric envelope is onto. This is the reason why we sometimes also refer to $A\to B$ as a quantum subgroup: The contravariance of going from (quantum) groups to functions on them means that surjections of algebras correspond to embeddings of their ``dual'' objects. 
\end{remark}

Given a $C^*$ quantum action $\rho_A:\cC(X)\to\cC(X)\otimes A$, one possible way to go about constructing a (D)-isometric envelope would be to consider the quotient $A\to A/I$, where
\begin{equation}\label{eq.quasab_ideal}
 I=\bigcap_\pi\ker(\pi)\text{ for }\pi:\rho_A\to\rho_B\text{ in }\cact\text{ with (D)-isometric }\rho_B. 
\end{equation}
This is more problematic that it might seem. The issue is that it's unclear whether $A/I$ inherits a quantum group structure $\Delta:A/I\to (A/I)\otimes (A/I)$ from $A$ compatible with $\Delta_A$ via the projection $A\to A/I$. 

While attempting to construct (Lip$_1$)-isometric envelopes (defined by analogy to \Cref{def.iso_env}), this problem led the authors of \cite{QuaSab12} to introduce a stronger notion of isometry that is formally, at least, stronger than (Lip$_1$), for which it could be shown that the analogous quotient $A\to A/I$ \define{is} a quantum subgroup \cite[Theorem 5.12]{QuaSab12}. 

Actions satisfying this stronger property are called \define{full-isometric} in \cite[Definition 5.8]{QuaSab12}. It seems to be unknown whether in fact (Lip$_1$) is equivalent to being full-isometric. The aim of this section is to show that this issue can be gotten around in the (D) version of the problem. 

The main result of the section is:

\begin{theorem}\label{th.sbgp}
A faithful $C^*$ quantum action $\rho:\cC(X)\to\cC(X)\otimes A$ has a (D)-isometric envelope. 
\end{theorem}

Within the hypotheses of the theorem, we begin by isolating the set of ``good'' states on $A$, i.e. those that satisfy the relevant isometry condition:

\begin{definition}\label{def.D_iso_state}
A functional $\psi\in A^*$ is \define{(D)-isometric} if the equality in \Cref{eq.D_isometric} holds upon applying $\psi$ to both sides:
\begin{equation}\label{eq.sbgp_D}
(x\triangleleft\psi)(d_y) = \left(y\triangleleft\overline{\psi}\right)(d_x),\ \forall x,y\in X,
\end{equation}
where $\overline{\psi}=\psi\circ\kappa$. 

We denote the set of (D)-isometric functionals by $A^*_D$.
\end{definition}

\begin{lemma}\label{le.A^*_D}
In the setting of \Cref{th.sbgp}, the subset $A^*_D\le A^*$ is a weak$^*$-closed linear subspace. 
\end{lemma}
\begin{proof}
Condition \Cref{eq.sbgp_D} is easily seen to be linear in $\psi$, hence the claim that $A^*_D$ is a linear subspace of $A^*$.  

To verify the closure claim, note that the map $A^*\to\cM(X)$ defined by $\psi\mapsto x\triangleleft\psi$ is continuous, being the dual of the $C^*$-algebra map
\begin{center}
\begin{tikzpicture}[auto]
  \node (1) at (0,0) {$\scriptstyle \cC(X)$};
  \node (2) at (2,0) {$\scriptstyle \cC(X)\otimes A$};
  \node (3) at (4,0) {$\scriptstyle A$.};
  \draw[->] (1) to node {$\scriptstyle \rho$} (2);
  \draw[->] (2) to node {$\scriptstyle x\otimes\id$} (3);
  %\draw[->,bend right=30] (1) to node [swap] {$\scriptstyle\mu\triangleleft\varphi$} (3);
\end{tikzpicture}
\end{center}  
Since $\psi\mapsto\overline{\psi}$ is also weak$^*$ continuous, it follows that \Cref{eq.sbgp_D} is a closed condition in $\psi$. 
\end{proof}

In fact, we have more:

\begin{lemma}\label{le.A^*_D_mult}
The subspace $A^*_D\le A^*$ is closed under the multiplication induced by $\Delta_A:A\to A\otimes A$. 
\end{lemma}
\begin{proof}
First, we need to recast \Cref{eq.sbgp_D} in a more convenient form. Averaging $x$ in that equation over $\mu\in\pr(X)$ and $y$ over $\nu\in\pr(X)$ we get 
\begin{equation}\label{eq.sbgp_D_bis}
(\mu\triangleleft\psi)(d_\nu) = \left(\nu\triangleleft\overline{\psi}\right)(d_\mu),
\end{equation}
where $d_\mu=\int_Xd_x\ \text{d}\mu$ is a $\cC(X)$-valued integral. The (D)-isometry condition \Cref{eq.sbgp_D} is equivalent to having \Cref{eq.sbgp_D_bis} for all probability measures $\mu,\nu$ on $X$.  

Now fix (D)-isometric functionals $\psi_i$, $i=1,2$ on $A$. We have to show that $\psi=\psi_1\psi_2$ satisfies \Cref{eq.sbgp_D} for any choice of $x,y\in X$. Consider the following computation:
\begin{multline}\label{mult.A^*_D_mult}
(x\triangleleft\psi_1\psi_2)(d_y) = ((x\triangleleft\psi_1)\triangleleft\psi_2)(d_y) \\ = \left(y\triangleleft\overline{\psi_2}\right)(d_{x\triangleleft\psi_1}) = \int_X d_{x\triangleleft\psi_1}(z)\ \text{d}\left(y\triangleleft\overline{\psi_2}\right)(z)\\ = \int_{X\times X}d(z,z')\ \text{d}(x\triangleleft\psi_1)(z)\ \text{d}\left(y\triangleleft\overline{\psi_2}\right)(z);
\end{multline}
the second equality uses the (D)-isometry of $\psi_2$ in the generalized form \Cref{eq.sbgp_D_bis} for $\mu=x\triangleleft\psi_1$ and $\nu=y$. 

The end result of \Cref{mult.A^*_D_mult} is invariant under the substitutions $x\leftrightarrow y$ and $\psi_i\leftrightarrow \overline{\psi_j}$ for $\{i,j\}=\{1,2\}$, so we would have obtained the same result had we started with $\left(y\triangleleft\overline{\psi_2}\overline{\psi_1}\right)$. Since the antipode $\kappa$ is comultiplication-reversing on $A$, the map $\varphi\mapsto\overline{\varphi}$ is multiplication-reversing on $A^*$. Hence, 
\[
 (x\triangleleft\psi_1\psi_2)(d_y) = \left(y\triangleleft\overline{\psi_1\psi_2}\right)(d_x),\ \forall x,y\in X. 
\]
This shows that $\psi_1\psi_2$ is (D)-isometric, as desired.   
\end{proof}

Multiplication on $A$ induces a left and a right action of $A$ on $A^*$ defined by 
\[
 (a,\psi) \mapsto a\psi = \psi(\bullet a)
\]
and respectively
\[
 (\psi,a) \mapsto \psi a = \psi(a\bullet)
\]
for $a\in A$ and $\psi\in A^*$, where $\bullet$ is a placeholder for a variable in $A$. 

Since $\psi\mapsto a\psi$ and $\psi\mapsto \psi a$ are weak$^*$ continuous operators on $A^*$ it follows from \Cref{le.A^*_D} that the space
\begin{equation}\label{eq.A^*_D_closure}
 \{\psi\in A^*\ |\ a\psi,\psi a\in A^*_D,\ \forall a\in A\}
\end{equation}
is weak$^*$-closed.

There is an inclusion-reversing one-to-one correspondence

\begin{equation}\label{eq.corresp}
\begin{tikzpicture}[auto,baseline=(current  bounding  box.center)]
  \node[text width=3cm] (1) at (0,0) {closed subspaces of $A$};
  \node[text width=4cm] (2) at (8,0) {weak$^*$-closed subspaces of $A^*$};
  \draw[->] (1) .. controls (2,1) and (4,1) .. node[text width=3.5cm]{$I\mapsto$ elements of $A^*$ annihilated by $I$} (2);
  \draw[<-] (1) .. controls (2,-1) and (4,-1) .. node[text width=3.5cm,below] {$W\mapsto$ elements of $A$ annihilated by $W$} (2);
\end{tikzpicture}
\end{equation}

Bilateral ideals in $A$ correspond to subspaces of $A^*$ invariant under the left and right actions by $A$. Since the weak$^*$-closed subspace \Cref{eq.A^*_D_closure} of $A^*$ is by definition left and right $A$-invariant, \Cref{eq.corresp} associates to it a closed ideal $I$ in $A$.

We henceforth denote by $T:A\to B=A/I$ the quotient $C^*$-algebra we have just introduced. It will turn out that $T$ is the underlying map of our (D)-isometric envelope of $\rho$. We work our way towards proving this through a series of auxiliary results.

\begin{lemma}\label{le.B^*_alg}
The subspace $B^*\le A^*$ is closed under multiplication. 
\end{lemma}
\begin{proof}
$B^*$ is the annihilator
\[
 \mathrm{Ann}(I)=\{\psi\in A^*\ |\ \psi(I)=0\},
\] 
(i.e. the subspace of $A^*$ corresponding to $I$ through \Cref{eq.corresp}); it is also the space defined by \Cref{eq.A^*_D_closure}.

Let $\psi_i\in B^*$ for $i=1,2$, and denote $\psi=\psi_1\psi_2$. We have to show that for every $a\in A$, both $a\psi$ and $\psi a$ are elements of $A^*_D$. Let us focus on $\psi a$.

Just running through the definitions we get
\begin{equation}\label{eq.B^*_alg}
 \psi a = (\psi_1\psi_2)a = ((\psi_1\otimes\psi_1)\Delta(a))\circ\Delta:A\to\bC,
\end{equation}
where the juxtaposition $(\psi_1\otimes\psi_1)\Delta(a)$ signifies the right action of an element $\Delta(a)\in A\otimes A$ on a functional $\psi_1\otimes\psi_2:A\otimes A\to\bC$.

Since the algebraic tensor square of $A$ is dense in $A\otimes A$, we can find a sequence of finite sums $a_n=\sum a_{n,1}\otimes a_{n,2}, n\in\bZ_{\ge 0}$ converging to $a$ with $a_{n,i}\in A$ (we have suppressed summation indices to streamline the notation). But then \Cref{eq.B^*_alg} is the weak$^*$ (and in fact even norm) limit of 
\begin{equation}\label{eq.B^*_alg_bis}
 \left((\psi_1\otimes\psi_2)\left(\sum a_{n,1}\otimes a_{n,2}\right)\right)\circ\Delta = \sum (\psi_1 a_{n,1})(\psi_2 a_{n,2}). 
\end{equation}
The factors $\psi_i a_{n,i}$ on the right hand side of \Cref{eq.B^*_alg_bis} are in $A^*_D$ because $\psi_i$ are elements of $B^*=$ \Cref{eq.A^*_D_closure}. Since $A^*_D\le A^*$ is a subalgebra by \Cref{le.A^*_D_mult}, \Cref{eq.B^*_alg_bis} is in $A^*_D$. Finally, the fact that $A^*_D$ is closed in $A^*$ implies that the limit $\psi a$ of \Cref{eq.B^*_alg_bis} as $n\to\infty$ is also an element of $A^*_D$. 
\end{proof}

\begin{lemma}\label{le.B_is_qg}
The quotient $B=A/I$ has a coassociative comultiplication $\Delta_B:B\to B\otimes B$ making it into a quantum group, and such that $T$ is a morphism of quantum groups.  
\end{lemma}
\begin{proof}
All we have to do is show that the composition
\begin{equation}\label{eq.prefact_Delta}
\begin{tikzpicture}[auto,baseline=(current  bounding  box.center)]
  \node (1) at (0,0) {$\scriptstyle A$};
  \node (2) at (2,0) {$\scriptstyle A\otimes A$};
  \node (3) at (4,0) {$\scriptstyle B\otimes B$};
  \draw[->] (1) to node {$\scriptstyle \Delta_A$} (2);
  \draw[->] (2) to node {$\scriptstyle T\otimes T$} (3);
  %\draw[->,bend right=30] (1) to node [swap] {$\scriptstyle\mu\triangleleft\varphi$} (3);
\end{tikzpicture}
\end{equation}  
factors as 
\begin{equation}\label{eq.fact_Delta}
\begin{tikzpicture}[auto,baseline=(current  bounding  box.center)]
  \node (1) at (0,0) {$\scriptstyle A$};
  \node (2) at (2,0) {$\scriptstyle A\otimes A$};
  \node (3) at (4,0) {$\scriptstyle B\otimes B$;};
  \node (4) at (2,-1) {$\scriptstyle B$};
  \draw[->] (1) to node {$\scriptstyle \Delta_A$} (2);
  \draw[->] (2) to node {$\scriptstyle T\otimes T$} (3);
  \draw[->,bend right=30] (1) to node [swap] {$\scriptstyle T$} (4);
  \draw[->,bend right=30,dashed] (4) to node [swap] {$\scriptstyle \Delta_B$} (3);
  %\draw[->,bend right=30] (1) to node [swap] {$\scriptstyle\mu\triangleleft\varphi$} (3);
\end{tikzpicture}
\end{equation}  
everything else (coassociativity of $\Delta_B$, the density property (2) from \Cref{def.cqg} and the fact that $T$ intertwines $\Delta_A$ and $\Delta_B$) will follow.

We know from \Cref{le.B^*_alg} that $B^*$ is a sub-ring of $A^*$. The functional $\psi_1\psi_2$ produced by composing \Cref{eq.prefact_Delta} further with $\psi_1\otimes\psi_2$ for $\psi_i\in B^*$ is then an element of $B^*$ and hence vanishes on $I$. Since the algebraic tensor product $B^*\otimes B^*$ separates $B\otimes B$ (e.g. \cite[Theorem III.4.9 (ii)]{Tak02}), it follows that \Cref{eq.prefact_Delta} vanishes on $I$ and hence factors as \Cref{eq.fact_Delta}.   
\end{proof}

\begin{proof of sbgp}  
As announced above, $T:A\to B$ will induce the desired (D)-isometric envelope $T:\rho\to\rho_B$, where $\rho_B$ is the composition
\begin{center}
\begin{tikzpicture}[auto,baseline=(current  bounding  box.center)]
  \node (1) at (0,0) {$\scriptstyle \cC(X)$};
  \node (2) at (2,0) {$\scriptstyle \cC(X)\otimes A$};
  \node (3) at (5,0) {$\scriptstyle \cC(X)\otimes B$.};
  \draw[->] (1) to node {$\scriptstyle \rho$} (2);
  \draw[->] (2) to node {$\scriptstyle \id\otimes T$} (3);
\end{tikzpicture}
\end{center}
To see this, consider a morphism $S:\rho\to\rho_C$ in $\cact$, where $\rho_C:\cC(X)\to \cC(X)\otimes C$ is a (D)-isometric action. All we have to do is show that the underlying map $S:A\to C$ factors uniquely through $T$. 

Uniqueness is immediate from the fact that $T$ is onto, so the important part is the existence of a factorization
\begin{center}
\begin{tikzpicture}[auto,baseline=(current  bounding  box.center)]
  \node (1) at (0,0) {$\scriptstyle A$};
  \node (2) at (2,0) {$\scriptstyle B$};
  \node (3) at (4,0) {$\scriptstyle C$.};
  \draw[->] (1) to node {$\scriptstyle T$} (2);
  \draw[->,dashed] (2) to node {} (3);
  \draw[->,bend right=30] (1) to node [swap] {$\scriptstyle S$} (3);
\end{tikzpicture}
\end{center}  
To prove this existence, note first that the dual $S^*$ maps $C^*$ into $A^*_D$. Indeed, $S^*$ is compatible with the actions of $C^*$ and $A^*$ on $\pr(X)$ in the sense that $x\triangleleft S^*(\psi)=x\triangleleft\psi$ for every $\psi\in C^*$; since all $\psi\in C^*$ are (D)-isometric, so are all functionals $S^*(\psi)$. 

Moreover, we for each $a\in A$ and $\psi\in C^*$ we have $aS^*(\psi) = S^*(S(a)\psi)$ (and similarly for right actions). We've just seen that the right hand side of this equality lies in $A^*_D$; since this happens for arbitrary $a$, it follows that $S^*(\psi)$ belongs to $B^*=$ \Cref{eq.A^*_D_closure}. We are now done: The dual of $S$ factors through $B^*\le A^*$, so $S$ must factor through $T:A\to B$.  
\end{proof of sbgp}

I do not know whether all quantum actions have (Lip$_p$)-isometric envelopes. Nevertheless, we now know that they do when the acting quantum group is finite-dimensional.

\begin{corollary}\label{cor.sbgp}
If the quantum group $A$ acting on $X$ via $\rho$ is finite-dimensional, then $\rho$ has a (Lip$_p$)-isometric envelope for every $p\ge 1$.

All of these envelopes coincide with the (D)-isometric one from \Cref{th.sbgp}. 
\end{corollary} 
\begin{proof}
Let $S:\rho\to \rho_C$ be a map in $\cact$ to a (Lip$_p$)-isometric action by $C$. The finite-dimensionality assumption implies that $S$ factors through an action $\rho_{C'}$ by some finite-dimensional quantum subgroup $C'\le C$. Moreover, $\rho_{C'}$ itself will be by necessity (Lip$_p$)-isometric and hence, by \Cref{pr.main_conv}, also (D)-isometric. But now the universality property of the (D)-isometric envelope implies that $S$ factors (uniquely) through it, as desired. 
\end{proof}

\begin{remark}
Of course, the same holds when $A$ is arbitrary but $X$ is finite, since in that case all notions of isometric action on $X$ coincide. 
\end{remark}

The idea behind the proof of \Cref{th.sbgp} is a familiar one: The first subspace $A^*_D\le A^*$ we isolated was weak$^*$ and multiplicatively closed, so it should be though of as dual to a ``quotient coalgebra'' of $A$, whose coaction on $\cC(X)$ has the desired property. We then had to slim down $A^*_D$ further to $B^*$ in order to pass to a ``Hopf algebra'' quotient of this preliminary ``coalgebra''.

The quotes are meant to indicate that the objects in question are not literally a coalgebra and Hopf algebra respectively, but rather functional-analytic counterparts (they are complete with respect to certain norms, the tensor products involved are completed, etc.). Nevertheless, one might hope to replicate the pattern to construct (Lip$_P$)-isometric envelopes.   

To do this, we would have to again isolate a weak$^*$ and multiplicatively closed set of functionals in $A^*$ satisfying the (Lip$_p$) condition. The problem is that this condition is phrased in terms of states rather than functionals. The best we can do is

\begin{definition}\label{def.Lipp_iso_state}
A state $\psi\in S(A)$ is \define{(Lip$_p$)-isometric} if 
\begin{equation}\label{eq.Lipp_iso_state}
W_p(x\triangleleft\psi,y\triangleleft\psi)\le d(x,y),\ \forall x,y\in X. 
\end{equation}

We denote the set of (Lip$_p$)-isometric states by $S_p(A)$. 
\end{definition}

It turns out that $S_p(A)$ is a compact convex sub-semigroup of $S(A)$, and hence its weak$^*$-closed linear span is a subalgebra of $A^*$ which we could use as the analogue of $A^*_D$ in the proof of \Cref{th.sbgp}. However, there might be, in principle, non-(Lip$_p$)-isometric states in this subspace. It is tempting to conjecture that this cannot happen:

\begin{conjecture}
Let $A$ act faithfully on the compact metric space $(X,d)$, and $p\ge 1$. 

The set of states in the weak$^*$-closed linear span of the set $S_p(A)$ from \Cref{def.Lipp_iso_state} coincides with $S_p(A)$. 
\end{conjecture}

If this were true, the proof of \Cref{th.sbgp} would go through and (Lip$_p$)-isometric envelopes would exist for all faithful $C^*$ quantum actions.

%%%%%%%%%%%%%%%%%%%%%%%%%%%%%%%%%%%%%%%%%%%%%%%%%%%%%%%%%%%%%%%%%%%%%%%%%%%%%%%%%%%%%%%%%%%%%%%%%%%%%%%%%%%%%%%%%%
%%%%%%%%%%%%%%%%%%%%%%%%%%%%%%%%%%%%%%%%%%%%%%%%%%%%%%%%%%%%%%%%%%%%%%%%%%%%%%%%%%%%%%%%%%%%%%%%%%%%%%%%%%%%%%%%%%

%%%%%%%%%%%%%%%%%%%%%%%%%%%%%%%%%%%%%%%%%%%%%%%%%%%%%%%%%%%%%%%%%%%%%%%%%%%%%%%%%%%%%%%%%%%%%%%%%%%%%%%%%%%%%%%%%%
%%%%%%%%%%%%%%%%%%%%%%%%%%%%%%%%%%%%%%%%%%%%%%%%%%%%%%%%%%%%%%%%%%%%%%%%%%%%%%%%%%%%%%%%%%%%%%%%%%%%%%%%%%%%%%%%%%

%\bibliography{qMetric}{}
\bibliographystyle{plain}
\addcontentsline{toc}{section}{References}

\def\polhk#1{\setbox0=\hbox{#1}{\ooalign{\hidewidth
  \lower1.5ex\hbox{`}\hidewidth\crcr\unhbox0}}}

\end{document}